\documentclass{amsart}
\usepackage[utf8]{inputenc}
\usepackage[T1]{fontenc}
\usepackage{geometry}
\usepackage{amssymb}
\numberwithin{equation}{section}

\def\a{\alpha}

\def\B{{\mathbb{B}}}
\def\C{\mathbb{C}}
\def\R{\mathbb{R}}
\def\N{\mathbb{N}}

\def\S{{\mathbb{S}}}

\def\M{{\mathbb{M}}}
\def\H{{\mathbb{H}}}

\newtheorem{thm}{Theorem}[section]
\newtheorem{lemma}[thm]{Lemma}

\newtheorem{rem}[thm]{Remark}
\newtheorem{prop}[thm]{Proposition}

\newtheorem*{thA}{Theorem A}
\newtheorem*{thB}{Theorem B}
\newtheorem*{thC}{Theorem C}

\begin{document}
	
	\title[$L^p-L^q$ ESTIMATES]{$L^p-L^q$ ESTIMATES OF BERGMAN PROJECTOR  ON THE MINIMAL BALL}
		\author{Jocelyn Gonessa}
	\address{Current address: Universit\'e de Bangui, D\'epartement 
		de math\'ematiques et Informatique, BP.908 Bangui-R\'epublique Centrafricaine}
	\email{gonessa.jocelyn@gmail.com}
	\email{jocelyn@aims.ac.za}
	
	\thanks{Gonessa was supported by \textit{International Centre of Theoretical Physics in Italia}}
	
	\begin{abstract}
		We study the $L^p-L^q$ boundedness of Bergman projector on the minimal ball. This improves an important result of \cite{MY} due to G. Mengotti and E. H. Youssfi .		
			\end{abstract}
	
	\maketitle
	
	\section{Introduction}
	G. Mengotti and E. H. Youssfi studied in \cite{MY} the $L^p-$ boundedness of Bergman projector on the minimal ball. Here we improve their result by giving the $L^p-L^q$ boundedness of Bergman projector. \textit{The minimal ball} $\B_*$ in $\C^n$ is defined as follows.
	$$\B_*=\{z\in\C^n:|z|^2+|z\bullet z|<1\}$$
	where $z\bullet w=\overset{n}{\underset{j=1}{\sum}} z_j w_j$ for $z$ and $w$ in $\C^n$. This is the unit ball of $\C^n$ with respect to the norm $N_*(z):=\sqrt{|z|^2+|z\bullet z|}$. The norm $N:=N_*/\sqrt2$ was introduced by Hahn and Pflug in \cite{HP}, where it was 
	shown to be the smallest norm in $\C^n$ that extends the euclidean norm in $\R^n$. More precisely, if $N$ is any complex norm in $\C^n$ such that 
	$N(x)=|x|=\sqrt{\sum_{j=1}^n x_j^2}$ for $x\in\R^n$ and $N(z)\leq |z|$ for $z\in\C^n$, then $N_*(z)/\sqrt{2}\leq N(z)$
	for $z\in \C^n$. Moreover, this
	norm was shown to be of interest in the study of several problems related to
	proper holomorphic mappings and the Bergman kernel, see for example \cite{G,GY,MY, M1}. The domain $\B_*$ is the first bounded domain in $\C^n$ which is neither Reinhardt nor homogeneous, and for which we have an explicit formula for its Bergman kernel.	The study of $L^p-L^q$ estimates of Bergman projector on smooth homogeneous is rather well understood on the unit ball and Siegel domains, see for example \cite{RZ}, \cite{Z3}, \cite{NS}, etc. 
		
	The authors of \cite{MY} developped a method for $L^p-$boundedness of Bergman projector on the minimal ball. Their argue consists to study the boundedness on an auxiliary complex manifold $\M$. Next, to transfer the results obtained on $\M$ to $\B_*$ via a proper holomorphic mapping. Our strategy combine the method of \cite{MY}, \cite{RZ} and a new ingredient.
	
	The plane of our research is the following. We first study the boundedness of certain class of integral operators on $\M$ by using the generalized Schur's test (see \cite{RZ}) and the Forelli-Ruding estimates (see \cite{MY}). As consequence we obtain the Bergman projector estimate in $\M$. Second we transplant the results obtained on $\M$ to Bergman projector on the minimal ball.
	
	\section{Preleminaries}
	We first define the auxiliary complex manifold $\M$. Let $ n\geq 2$ and consider the nonsingular cone
	$$\H:=\{z\in\C^{n+1}: \ z_1^2+\cdots+z_{n+1}^2=0,\,\,z\neq 0\}.$$
	This  is the orbit of the vector $(1, i, 0,  \ldots, 0)$ under the $SO(n+1, \C)-$action on $\C^{n+1}.$ 
	It is well-known that $\H$ can be identified with  the cotangent bundle of the unit sphere $\S^n$ in 
	the $n-$dimensional sphere in $\R^{n+1}$ minus its zero section. It was proved in \cite{OPY} that there 
	is a unique (up to a multiplicative constant) $SO(n+1,\,\,\C)-$invariant holomorphic form $\alpha$ on $\H$. 
	The restriction of this form to $\H\cap (\C\backslash\{0\})^{n+1}$ is given by
	$$\alpha(z)=\sum_{j=1}^{n+1}\frac{(-1)^{j-1}}{z_j}dz_1\wedge\cdots\wedge
	\widehat{dz_j}\wedge\cdots\wedge dz_{n+1}.$$
	The complex manifold $\M$ is defined by
	$$\M=\{z\in\H:\,\,\,|z|<1\}$$
	The orthogonal group $O(n + 1,\R)$ acts transitively on the manifold $$\partial \M=\{z\in\C^{n+1}:\,\,z\bullet z=0\,\,\textrm{and}\,\,|z|=1\}$$ Thus there is a unique $O(n + 1,\R)-$invariant probability
	measure $\mu$ on $\M$. This measure is induced by the Haar probability
	measure of $O(n + 1,\R)$ (see \cite{MY}). For any $\C^\infty-$ function $f$ on $\H$ we have, from \cite[Lemma 2.1]{MY},  that 
	\begin{equation}\label{intMY}
	\int_\H f(z)\a(z)\wedge\overline{\a(z)}=m_n\int^\infty_0 t^{2n-3}\int_{\partial\M}f(t\xi)d\mu(\xi)dt
	\end{equation}
	provided that the integrals make sense. Moreover
	$$m_n=2(n-1)\int_{z\in\M} \a(z)\wedge\overline{\a(z)}.$$
	For all $0<p<\infty$ we consider Lebesgue space $L_s^p(\M)$ on the measure space $(\M,(1-|z|^2)^s\a(z)\wedge\overline{\a(z)})$. The Bergman space $A_s^p(\M)$ is the subspace of $L_s^p(\M)$ consisting of holomorphic functions. $A_s^2(\M)$ is the closed subspace of the Hilbert space $L_s^2(\M)$. There exists a unique orthogonal projection $L_s^2(\M)$ onto $A_s^2(\M)$. That is the weighted Bergman projection. Its explicit expression is the following.
	$$P_{s,\M}f(z)=\int_{\M}K_{s,\M}(z,w)f(w)(1-|w|^2)^s\a(w)\wedge\overline{\a(w)}$$
	where the so called kernel Bergman $K_{s,\M}$ (see \cite[Theorem 3.2]{MY}) is given by 
	$$K_{s,\M}(z,w)=\frac{C\left(n-1+(n+1+2s)z\bullet\bar w\right)}{(1-z\bullet\bar w)^{n+1+s}}.$$
	Here $C$ is a certain constant that depends on $n$ and $s$. In this paper we consider the class of operators defined as follow.
	$$S_{\M}f(z)=(1-|z|^2)^{b_1}\int_{\M}
	\frac{f(w)}{(1-z\bullet\bar w)^c}(1-|w|^2)^{b_2}\a(w)\wedge\overline{\a(w)}$$ 
	and 
	$$T_{\M}f(z)=(1-|z|^2)^{b_1}\int_{\M}
	\frac{f(w)}{|1-z\bullet\bar w|^c}(1-|w|^2)^{b_2}\a(w)\wedge\overline{\a(w)}$$ 
	where $b_1$, $b_2$ and $c$ are any real numbers. 
	
	\section{ Statement of the auxiliaries results}
	The following auxiliaries results will play a key role for proving the main result of the paper.
	
	\begin{thA}\label{mainthA}{\it
			Let $b_1$, $b_2$ and $c$ be real numbers. Let $1<p\leq q<\infty$, $\max(-1,-1-qb_1)<r<\infty$ and $-1<s<\infty$. 
			Then the following assertions are equivalent.
			\begin{enumerate}
				\item[(i)] The operator $T_{\M}$ is bounded from $L_s^p(\M)$ to $L_r^q(\M)$.
				\item[(ii)] The operator $S_{\M}$ is bounded from $L_s^p(\M)$ to $L_r^q(\M)$.
				\item[(iii)] The parameters satisfy\\
				
				$
				\left\lbrace
				\begin{array}{ll}
				s+1<p(b_2+1)\\
				c\leq b_1+b_2-s+\frac{n+1+r}{q}+\frac{n+1+s}{p'}
				\end{array}
				\right.
				$
				
			\end{enumerate}
		}
	\end{thA}

	\begin{thB}\label{mainthB}{\it
			Let $b_1$, $b_2$ and $c$ be real numbers. Let $1=p\leq q<\infty$, $\max(-1,-1-qb_1)<r<\infty$ and $-1<s<\infty$. 
			Then the following assertions are equivalent.
			\begin{enumerate}
				\item[(i)] The operator $T_{\M}$ is bounded from $L_s^1(\M)$ to $L_r^q(\M)$.
				\item[(ii)] The operator $S_{\M}$ is bounded from $L_s^1(\M)$ to $L_r^q(\M)$.
				\item[(iii)] The parameters satisfy\\
				
				$
				\left\lbrace
				\begin{array}{ll}
				s<b_2\\
				c= b_1+b_2-s+\frac{n+1+r}{q}
				\end{array}
				\right.
				$
				or
				$			
				\left\lbrace
				\begin{array}{ll}
				s\leq b_2\\
				c<b_1+b_2-s+\frac{n+1+r}{q}
				\end{array}
				\right.
				$
			\end{enumerate}
		}
	\end{thB}

	\section{Statement of the main result}
	To state the main result we need the following definitions. For any $s>-1$ we define the weighted measure $v_s$ on $\B_*$ by $dv_s(z)=(1-N^2_*(z))^sdv(z)$ where $v$ is the normalized Lebesgue measure on $\B_*$. For all $0<p<\infty$ we consider Lebesgue space $L_s^p(\B_*)$ on the measure space $(\B_*,|z\bullet z|^\frac{p-2}{2}dv_s)$. The Bergman space $A_s^p(\B_*)$ is the subspace of $L_s^p(\B_*)$ consisting of holomorphic functions. It is well-known for $p=2$ there exists a unique orthogonal projection from $L_s^2(\B_*)$ onto $A_s^2(\B_*)$. That is so called weighted Bergman projection and denoted $P_{s,\B_*}$. Also, it is well-known that $P_{s,\B_*}$ is an integral operator on $L_s^2(\B_*)$. More precisely
	$$P_{s,\B_*}f(z)=\int_{\B_*}K_{s,\B_*}(z,w)f(w)dv_s(w)$$
	and the so called Bergman kernel $K_{s,\B_*}$ is explicitly given in \cite[Theorem A]{MY} by
	$$K_{s,\B_*}(z,w)=\frac{1}{(n^2+n-s)v_s(\B_*)}\frac{A(1-z\bullet\bar w,z\bullet z\overline{w\bullet w})}{\left((1-z\bullet w)^2-z\bullet z\overline{w\bullet w}\right)^{n+1+s}}$$ 
	where
	\begin{eqnarray*}
		A(X,Y)&=&\sum_{k=0}^{\infty}
		\left(
		\begin{array}{c}
			n+s+1\\
			2k+1
		\end{array}
		\right)
		X^{n+s-2k-1}Y^k\\
		&&\times \left[2(n+s)-\frac{(n+1+2s)(n+s-2k)}{(n+s+1)}(X^2-Y)\right]
	\end{eqnarray*}
	with $$\left(
	\begin{array}{c}
	n+s+1\\
	2k+1
	\end{array}
	\right)=\frac{(n+s+1)(n+s)\cdots(n+s-2k+1)}{(2k+1)!}$$

	The main result of the paper is the following. 
	
	\begin{thC}\label{mainthC}{\it
			Let $1\leq p\leq q <\infty$, $-1<\lambda,\tilde{\lambda}<\infty$. 
			\begin{enumerate}
				\item[(i)] For $1< p\leq q <\infty$ the Bergman projector $P_{s,\B_*}$ is bounded from $L^p_{\lambda}(\B_*)$ into $A_{\tilde{\lambda}}^q(\B_*)$ if and only if  $
				\left\lbrace
				\begin{array}{ll}
				\lambda+1<p(s+1)\\
				s\geq \frac{n+1+\lambda}{p}- \frac{n+1+\tilde{\lambda}}{q}
				\end{array}
				\right.
				$ 
				\item[(ii)] For $1=p\leq q <\infty$ the Bergman projector $P_{s,\B_*}$ is bounded from $L_{\lambda}^1(\B_*)$ into $A_{\tilde{\lambda}}^q(\B_*)$ if and only if 
				
				$
				\left\lbrace
				\begin{array}{ll}
				\lambda<s\\
				\frac{n+1+\tilde{\lambda}}{q}\geq n+1+\lambda
				\end{array}
				\right.
				$
				or
				$			
				\left\lbrace
				\begin{array}{ll}
				\lambda\leq s\\
				\frac{n+1+\tilde{\lambda}}{q}>n+1+\lambda
				\end{array}
				\right.
				$
				
			\end{enumerate}

		}
	\end{thC}

To prove our results we need the following results.
\begin{lemma}\cite[Lemma 5.1]{MY}\label{Jlemma}
	Let $d\in\N$. For $z\in\M$, $c\in\R$, $s> -1$, define
	$$I_c(z)=\int_{\partial\M}\frac{|z\bullet\xi|^{2d}}{|1-z\bullet\bar\xi|^{n+c}}d\mu(\xi)$$
	and 
	$$J_{c,s}(z)=\int_{\partial\M}\frac{|z\bullet w|^{2d}}{|1-z\bullet\bar w|^{n+c+s+1}}(1-|w|^2)^s\alpha(w)\wedge\overline{\alpha(w)}$$
When $c<0$, then $I_c$ and $J_{c,s}$ are bounded in $\M$. When $c>0$ then $I_c(z)\eqsim (1-|z|^2)^{-c}\eqsim J_{c,s}(z)$. Finally, $I_0(z)\eqsim \log\frac{1}{1-|z|^2}\eqsim J_{0,s}(z)$
\end{lemma}
\begin{rem}
	The symbol $u(z)\eqsim v(z)$ means that $u(z)/v(z)$ has finite limit as $|z|$ tends to $1$.
\end{rem}

The following results are the boundedness criterions for integral operators from $L^p$ into $L^q$ called generalize Schur's test.

\begin{thm}\cite[Theorem 1]{RZ}\label{Schur1}
	Let $\nu_1$ and $\nu_2$ be postive measures on the space $\Omega$ and let $K(z,w)$ be a non-negative measurable function on $\Omega\times \Omega$. Let $T$ be the integral operator with kernel $K$, defined as follows.
	$$Tf(z)=\int_{\Omega}f(w)K(z,w)d\nu_1(w)$$ 
	Suppose $1<p\leq q<\infty$, $\frac{1}{p}+\frac{1}{p'}=1$ and suppose there exist $\gamma$ and $\delta$ such that $\gamma+\delta=1$. If there exist positive functions $h_1$ and $h_2$ with positive constants $C_1$ and $C_2$
	$$\int_{\Omega}h_1(w)^{p'}K(z,w)^{\gamma p'}d\nu_1(w)\leq C_1h_2(z)^{p'}$$ 
	for almost all $z\in\Omega$, and
	$$\int_{\Omega}h_2(z)^{q}K(z,w)^{\delta p'}d\nu_2(z)\leq C_1h_1(w)^{q}$$ 
	for almost all $w\in\Omega$, then $T$ is bounded from $L^p(\Omega,\nu_1)$ into $L^q(\Omega,\nu_2)$ and the norm of the operator does not exceed $C_1^{1/p'}C_2^{1/q}$.
	\end{thm}

\begin{thm}\cite[Theorem 2]{RZ}\label{Schur2}
	Let $\nu_1$ and $\nu_2$ be postive measures on the space $\Omega$ and let $K(z,w)$ be a non-negative measurable function on $\Omega\times \Omega$. Let $T$ be the integral operator with kernel $K$, defined as follows.
	$$Tf(z)=\int_{\Omega}f(w)K(z,w)d\nu_1(w)$$ 
	Suppose $1=p\leq q<\infty$ and suppose there exist $\gamma$ and $\delta$ such that $\gamma+\delta=1$. If there exist positive functions $h_1$ and $h_2$ with positive constants $C_1$ and $C_2$ such that
	$$\textsl{ess}\sup_{w\in\Omega}
	h_1(w)K(z,w)^{\gamma}d\nu_1(w)\leq C_1h_2(z)$$ 
	for almost all $z\in\Omega$, and
	$$\int_{\Omega}h_2(z)^{q}K(z,w)^{\delta q}d\nu_2(z)\leq C_1h_1(w)^{q}$$ 
	for almost all $w\in\Omega$, then $T$ is bounded from $L^1(\Omega,\nu_1)$ into $L^q(\Omega,\nu_2)$ and the norm of the operator does not exceed $C_1C_2^{1/q}$.
	
\end{thm}

\section{Sufficient conditions for $L^p-L^q$ estimates of $S_\M$}
In this section the main ingredient is the generalize Schur's test. We are begining by the following lemma.

\begin{lemma}\label{lemm1}
	Let $b_1$, $b_2$, and $c$ be real numbers. Let $1<p\leq q<\infty$, $\max(-1,-1-qb_1)<r<\infty$ and $-1<s<\infty$. 
	If
	\begin{equation}\label{condition1}
		\left\lbrace
		\begin{array}{ll}
		s+1<p(b_2+1)\\
		c\leq b_1+b_2-s+\frac{n+1+r}{q}+\frac{n+1+s}{p'}
		\end{array}
		\right.
	\end{equation}
		then $S_{\M}$ is bounded from $L_s^p(\M)$ to $L_r^q(\M)$.	
	
\end{lemma}
\begin{proof}
To use generalize Schur's test we first consider the following tools.
$$h_1(z)=(1-|z|^2)^{-u},\,\, h_2(w)=(1-|w|^2)^{-v},$$ $$d\nu_1(z)=(1-|z|^2)^{s}\alpha(z)\wedge\overline{\alpha(z)},\,\, d\nu_2(w)=(1-|w|^2)^{r}\alpha(w)\wedge\overline{\alpha(w)}$$ and $$K(z,w)=\frac{(1-|z|^2)^{b_1}(1-|w|^2)^{b_2-s}}{|1-z\bullet\bar w|^{b_1+b_2+\frac{n+1+r}{q}-\frac{n+1+s}{p}}}.$$  
Second, observe that if $c\leq b_1+b_2+\frac{n+1+r}{q}-\frac{n+1+s}{p}$ then
$$|S_\M f(z)|\leq |T_\M f(z)|\leq 2^{b_1+b_2+\frac{n+1+r}{q}-\frac{n+1+s}{p}}|Tf(z)|.$$
Thus the boundedness of $S_\M$ arises from $T$ where
$$Tf(z)=\int_{\M}f(w)K(z,w)d\nu_1(w)$$
To do this we adopt the following notations.
\begin{equation}\label{cchoose}
c=b_1+b_2-s+\frac{n+1+s}{p'}+\frac{n+1+r}{q}
\end{equation}
\begin{equation}\label{tauchoose}
\tau=\frac{n+1+s}{p'}+\frac{n+1+r}{q}
\end{equation}
Let us choose 
\begin{equation}\label{tchoose}
t=\frac{\frac{n+1+s}{p'}+v-u}{\tau}
\end{equation}
where $u$ and $v$ will be determined. It is easy to see that 
\begin{equation}\label{1-tchoose}
1-t=\frac{\frac{n+1+r}{q}+u-v}{\tau}
\end{equation}
\begin{equation}\label{Jzchoose}
\int_{\M}h_1(w)^{p'}K(z,w)^{tp'}d\nu_1(w)=(1-|z|^2)^{b_1tp'}J_{c_1,s_1}
(z) 
\end{equation}
where $c_1=cp't+p'u-(b_2-s)tp'-n-s-1$ and $s_1=cp't+p'u-(b_2-s)tp'-n-s-1$,
\begin{equation}\label{Jwchoose}
\int_{\M}h_2(z)^{q}K(z,w)^{(1-t)q}d\nu_2(z)=(1-|w|^2)^{q(b_2-s)(1-t)}J_{c_2,s_2}
(w) 
\end{equation}
where $c_2=cq(1-t)+qv-b_1q(1-t)-n-r-1$ and $s_2=cq(1-t)+qv-b_1q(1-t)-n-r-1$. It is clear that from (\ref{cchoose}) and (\ref{tauchoose}) we get 
 \begin{equation}\label{ctauchoose}
 c-b_1-b_2+s=\tau.
 \end{equation}
 So, from (\ref{ctauchoose}) and (\ref{tauchoose}) we obtain that
  \begin{equation}\label{c1choose}
  b_1tp'-c_1=-p'v
 \end{equation}
 and
 \begin{equation}\label{c2choose}
 (b_2-s)(1-t)q-c_2=-qu
 \end{equation}
  Otherwise we claim that there exist two reals numbers $u$ and $v$ such that
  \begin{equation}\label{claim1}
  	s_1>-1,\,\,s_2>0,\,\,c_1>0,\,\,c_2>0.
  \end{equation}
  So, from Lemma \ref{Jlemma} combined with (\ref{Jzchoose}) and (\ref{Jwchoose}) we obtain that.
          \begin{equation}\label{Jchoose}
     \left\lbrace
     \begin{array}{ll}
     \int_{\M}h_1(w)^{p'}K(z,w)^{tp'}d\nu_1(w)\leq C_1(1-|z|^2)^{-p'v}\\
     \int_{\M}h_2(z)^{q}K(z,w)^{(1-t)q}d\nu_2(z)\leq C_2(1-|w|^2)^{-qu}
     \end{array}
     \right.
     \end{equation}
  We acheive the lemma's proof by invoking Theorem \ref{Schur1}. Now we are going to prove (\ref{claim1}). First, it is easy to see that (\ref{condition1}) yields the following.
       \begin{equation}\label{Jclaim1}
  \left\lbrace
  \begin{array}{ll}
  -\frac{(b_2-s)(n+1+r)}{q}<\frac{(1+s)\tau}{p'}+\frac{(b_2-s)(n+1+s)}{p'}\\
  -\frac{b_1(n+1+s)}{q}<\frac{b_1(n+1+r)}{q}+\frac{\tau(1+r)}{q}
  \end{array}
  \right.
  \end{equation} 
  Second, we choose $u$ and $v$ such that
  \begin{equation}\label{Jclaim2}
  \left\lbrace
  \begin{array}{ll}
  -\frac{(b_2-s)(n+1+r)}{q}<\tau u+(b_2-s)(u-v)<\frac{(1+s)\tau}{p'}+\frac{(b_2-s)(n+1+s)}{p'}\\
  -\frac{b_1(n+1+s)}{q}<\tau v+b_1(v-u)<\frac{b_1(n+1+r)}{q}+\frac{\tau(1+r)}{q}
  \end{array}
  \right.
  \end{equation}
  Finally, by combining (\ref{cchoose}), (\ref{tauchoose}), (\ref{tchoose}), (\ref{1-tchoose}) and (\ref{Jclaim2}) we prove easly (\ref{claim1}).
  
\end{proof}

\begin{lemma}\label{lemm2}
	Let $b_1$, $b_2$, and $c$ be real numbers. Let $1=p\leq q<\infty$, $\max(-1,-1-qb_1)<r<\infty$ and $-1<s<\infty$. 
	If
	\begin{equation}\label{condition2}
	\left\lbrace
	\begin{array}{ll}
	s+1<p(b_2+1)\\
	c\leq b_1+b_2-s+\frac{n+1+r}{q}+\frac{n+1+s}{p'}
	\end{array}
	\right.
	\end{equation}
		then $S_{\M}$ is bounded from $L_s^1(\M)$ to $L_r^q(\M)$.	
	
\end{lemma}
\begin{proof}
As in proof of the Lemma\ref{lemm1} we have that. $\tau=\frac{n+1+r}{q}$, $t=\frac{v-u}{\tau}$, $u<(b_2-s)t$ and $-tb_1<v$. From easy calculus we have
$$\max(\frac{1-|z|^2}{2},\frac{1-|w|^2}{2})\leq|1-z\bullet\bar w|$$. This yields the following.
\begin{align*}
	\sup_{w\in\M}h_1(w)K(z,w)^t&\leq 2^{tb_1+t(b_2-s)+v-u}(1-|z|^2)^{-v}\sup_{w\in\M}|1-z\bullet\bar w|^{tb_1+t(b_2-s)+v-u-tc}\\
	&\leq 4^{tb_1+t(b_2-s)+v-u-tc/2}(1-|z|^2)^{-v}
\end{align*}
Otherwise, using the same method in Lemma\ref{lemm1} it is obvious to prove that.
$$\int_{\M}h_2(z)^{q}K(z,w)^{(1-t)q}d\nu_2(z)\leq C_2h_1(w)^{q}$$	
Finally, the lemma arises from Theorem \ref{Schur2}.

\end{proof}
\begin{lemma}\label{lemm2.2}
	Let $b_1$, $b_2$, and $c$ be real numbers. Let $1=p\leq q<\infty$, $\max(-1,-1-qb_1)<r<\infty$ and $-1<s<\infty$. 
	If
	\begin{equation}\label{condition2.2}
	\left\lbrace
	\begin{array}{ll}
	s=b_2\\
	c\leq b_1+\frac{n+1+r}{q}
	\end{array}
	\right.
	\end{equation}
	then $S_{\M}$ is bounded from $L_s^1(\M)$ to $L_r^q(\M)$.	
	
\end{lemma}
\begin{proof}
	Here, we consider $h_1(z)=1$, $h_2(z)=(1-|z|^2)^{-v}$, $K(z,w)=\frac{(1-|z|^2)^{b_1}}{|1-z\bullet\bar w|^{c}}$ and $t=\frac{v}{c-b_1}$ where $c>0$ and $\frac{1+r}{q}+b_1(1-t)<v<c-b_1$. Then
	\begin{align*}
		\sup_{w\in\M}h_1(w)K(z,w)^t&=\sup_{w\in\M}\frac{(1-|z|^2)^{tb_1}}{|1-z\bullet\bar w|^{v+tb_1}}\\
		&\leq 2^{ct}(1-|z|^2)^{-v}
	\end{align*}
	Otherwise, from Lemma\ref{Jlemma} we get that. 
	\begin{align*}
		\int_{\M}h_2(z)^{q}K(z,w)^{(1-t)q}d\nu_2(z)&=J_{c_3,s_3}
		(w)\\
		&\leq C_2	
	\end{align*}
	where $c_3=(1-t)qc-b_1(1-t)q+qv-n-r-1<0$ and $s_3=(1-t)b_1q-qv+r>-1$. 
\end{proof}

\section{Necessary conditions for $L^p-L^q$ estimates of $T_\M$}

\begin{lemma}\label{lemm3}
	Let $b_1$, $b_2$, and $c$ be real numbers. Let $1\leq p\leq q<\infty$, $\max(-1,-1-qb_1)<r<\infty$ and $-1<s<\infty$. 
	If $T_{\M}$ is bounded from $L_s^p(\M)$ to $L_r^q(\M)$ then 
	\begin{equation}\label{condition3}
	\left\lbrace
	\begin{array}{ll}
	s+1\leq p(b_2+1)\\
	c\leq b_1+b_2-s+\frac{n+1+r}{q}+\frac{n+1+s}{p'}
	\end{array}
	\right.
	\end{equation}
	and the strict inequality holds for $1< p\leq q<\infty$.
	\end{lemma}

\begin{proof}
	Suppose $1< q<\infty$. Then the dual space $L^q_r(\M)^*$ of $L^q_r(\M)$ can be indentified with $L^{q'}_r(\M)$ under the integral paring
$$<f,g>_r=\int_{\M}f(z)\overline{g(z)}d\nu_2(z),\,f\in L^q_r(\M),\,g\in L^{q'}_r(\M)$$ 
Moreover, by easy computation we have
$$T_\M^*g(z)=(1-|z|^2)^{b_2-s}\int_{\M}\frac{(1-|w|^2)^{b_1}}{(1-z\bullet\bar w)^{c}}g(w)d\nu_2(w)$$
Let be $N$ a real number such that
\begin{equation}\label{Nreal}
 N>\max(-\frac{1+r}{q'},-1-r-b_1). 
 \end{equation}
 Then  from (\ref{intMY}) we have
\begin{equation}\label{fN}
\int_{\M}|f_N(z)|^{q'}d\nu_2(z)=\frac{\omega(\partial\M)\Gamma(r+q'N+1)\Gamma(n-1)}{2\Gamma(r+q'N+n)}
\end{equation}

\begin{equation}\label{T*fN}
T_\M^*f_N(z)=C_N(1-|z|^2)^{b_2-s}=C_Nf_{b_2-s}(z)
\end{equation}

where $$f_N(z)=(1-|z|^2)^{N}$$
 and 
$$C_N=\frac{(n-1)!\omega(\partial\M)\Gamma(b_1+r+N+1)}{(n-1)(n-2)!2\Gamma(c)}\sum_{k=0}^{\infty}\frac{\Gamma(k+c)\Gamma(n-1+k/2)}{\Gamma(b_1+r+N+n+k/2)}$$
Because $T_\M^*f_N$ belongs to $L^{p'}_s(\M)$ then from (\ref{T*fN}) and (\ref{Nreal}) we have $s+p'(b_2-s)>-1$. This leads to $s+1<p(b_2+1)$. Now, if we suppose $1=p<q<\infty$ then $T_\M^*f_N$ belongs to $L^{\infty}_s(\M)$. This gives $b_2-s\geq 0$. Thus $s+1\leq 1(b_2+1)$. The case $1=p=q$ is easy to prove. This completes the proof of the lemma.
	\end{proof}

\begin{lemma}\label{lemm4}
	Let $b_1$, $b_2$, and $c$ be real numbers. Let $1\leq p\leq q<\infty$, $\max(-1,-1-qb_1)<r<\infty$ and $-1<s<\infty$. 
	Suppose $T_{\M}$ bounded from $L_s^p(\M)$ to $L_r^q(\M)$. Consider the following tree cases.
	\begin{enumerate}
		\item[(i)] $1< p\leq q<\infty$ and $s+1<p(b_2+1)$;
		\item[(ii)] $1=p\leq q<\infty$ and $s<b_2$;
		\item[(iii)] $1=p\leq q<\infty$ and $s=b_2$.
	\end{enumerate}
	If $(i)$ and $(ii)$ hold then
	\begin{equation}\label{condition4} 
	c\leq b_1+b_2-s+\frac{n+1+r}{q}+\frac{n+1+s}{p'};
	\end{equation}
	 if $(iii)$ hold then 
	 \begin{equation}\label{condition5}
	 c< b_1+b_2-s+\frac{n+1+r}{q}+\frac{n+1+s}{p'}.
	 \end{equation}
\end{lemma}

\begin{proof}
For any $\xi\in \M$ we denote
$$f_\xi(z)=\frac{ (1-|\xi|^2)^{n+1+b_2}[n-1+(n+1+2b_2)z\bullet\bar{\xi}]}{(1-z\bullet\bar\xi)^{n+1+b_2}}$$
Then $(i)$ leads to
\begin{equation}\label{xi}
\int_{\M}|f_\xi(z)|^{p}d\nu_1(z)\leq 2(n+1+b_2)(1-|\xi|^2)^{p(n+1+b_2)-n-1-s}J_{p(n+1+b_2)-n-1-s,s}(\xi)
\end{equation}
Otherwise, because $g(\xi)=\frac{1}{(1-\xi\bullet\bar z)^{c}}$ belongs to $A_s^2(\M)$ we have that
\begin{align*}
T_\M(f_\xi)(z)&=(1-|z|^2)^{b_1}(1-|\xi|^2)^{n+1+b_2-(n+1+s)/p}\overline{P_\M(g)(\xi)}\\
&=\frac{(1-|z|^2)^{b_1}(1-|\xi|^2)^{n+1+b_2-(n+1+s)/p} }{(1-z\bullet\bar\xi)^{c}}
\end{align*}
From the boundedness of $T_\M$ we have that
 \begin{equation}\label{Txi}
 (1-|\xi|^2)^{q(n+1+b_2)-(n+1+s)q/p}J_{c_4,s_4}(\xi)\leq C
\end{equation}
where $C>0$, $c_4=qc-qb_1-n-1-r-q(n+1+b_2)+(n+1+s)q/p$ and $s_4=qb_1+r$. So, from Lemma \ref{Jchoose} we have $c_4\leq 0$. This gives (\ref{condition4}). By the same way the case $(ii)$ leads to (\ref{condition4}). The case $(iii)$ leads to 
$(n+1+s)q/p-q(n+1+b_2)=s-b_2=0$. So, from (\ref{Txi}) combined with Lemma \ref{Jchoose} we abtain easly (\ref{condition5}). 
\end{proof}

\section{Proof of the Theorem A}
\begin{proof}
	The assertion (i) implies (iii) follows from Lemma \ref{lemm3}. It is obvious that (ii) implies (i). The assertion (iii) implies (ii) follows from Lemma \ref{lemm1}. This completes the proof of the theorem. 
\end{proof}

\section{Proof of the Theorem B}
\begin{proof}
	The assertion (i) implies (iii) follows from (ii) of Lemma \ref{lemm3} and (iii) of Lemma \ref{lemm4}. It is obvious that (ii) implies (i). The assertion (iii) implies (ii) follows from Lemma \ref{lemm1} and Lemma \ref{lemm2.2}. This achieves the proof of the theorem. 
\end{proof}
Before proving Theorem C we recall the key tool which will be use.
Let $f\,:\, \B_*\mapsto\C$ be a measurable function. We define a function $I_\M f$ on $\M$
by
$$(I_\M f )(z)=\frac{z_ {n+1} f\circ F (z)}{(2(n + 1) )^ {1/p}}=\frac{z_ {n+1} f (z_1 ,\cdots, z_n )}{(2(n + 1) )^ {1/p}}
$$
\begin{lemma}\cite[Lemma 4.1]{MY}\label{Isometry}
For each $p\geq 1$ and $\lambda>-1$ the linear operator $I_\M$ is an isometry from $L^p_\lambda(\B_*)$ into $L^p_\lambda(\M)$. Moreover, we have $P_{\lambda,\M}I_\M=I_\M P_{\lambda,\B_*}$ on $L^\lambda_s(\B_*)$.  
	\end{lemma}

\begin{prop}\label{propo}{\it
		Let $1\leq p\leq q <\infty$, $-1<\lambda,\tilde{\lambda}<\infty$. 
		\begin{enumerate}
			\item[(i)] For $1< p\leq q <\infty$ the Bergman projector $P_{s,\M}$ is bounded from $L^p_{\lambda}(\M)$ onto $A_{\tilde{\lambda}}^q(\M)$ if and only if  $
			\left\lbrace
			\begin{array}{ll}
			\lambda+1<p(s+1)\\
			s\geq \frac{n+1+\lambda}{p}- \frac{n+1+\tilde{\lambda}}{q}
			\end{array}
			\right.
			$ 
			\item[(ii)] For $1=p\leq q <\infty$ the Bergman projector $P_{s,\M}$ is bounded from $L_{\lambda}^1(\M)$ onto $A_{\tilde{\lambda}}^q(\M)$ if and only if 
			
			$
			\left\lbrace
			\begin{array}{ll}
			\lambda<s\\
			\frac{n+1+\tilde{\lambda}}{q}\geq n+1+\lambda
			\end{array}
			\right.
			$
			or
			$			
			\left\lbrace
			\begin{array}{ll}
			\lambda\leq s\\
			\frac{n+1+\tilde{\lambda}}{q}>n+1+\lambda
			\end{array}
			\right.
			$
			
		\end{enumerate}

	}
\end{prop}

\begin{proof}
	Let us choose in Theorem A and B $c=n+1+\lambda$, $b_2=s$ and $b_1=0$. Then it follows from Theorem A that $P_{s,\M}$ is bounded from $L^p_\lambda(\M)$ onto $A^q_{\tilde{\lambda}}(\M)$ iff (iii) of Theorem A holds.  Otherwise, from Theorem B it follows that $P_{\lambda,\M}$ is bounded from $L^1_\lambda(\M)$ onto $A^q_{\tilde{\lambda}}(\M)$ iff (iii) of Theorem B holds. This achieves the proof of the proposition.	
\end{proof}

\begin{rem}
	 The assertion (i) of Proposition \ref{propo} improves Theorem 5.2 of \cite{MY}. Indeed, it suffices to take $\lambda=\tilde{\lambda}$ and $p=q\geq 1$.
\end{rem}

\section{Proof of the Theorem C}

  \begin{proof}
  	  	The equivalence of (i) and (iii) of Theorem A follows from Lemma \ref{Isometry}. Also, the equivalence of (ii) and (iii) of Theorem B follows from Lemma \ref{Isometry}. This completes the proof of the theorem.
  	
\end{proof}
 
\begin{rem}
	The assertion (i) of Theorem C improves an important result due to G. Mengotti and E. H. Youssfi \cite{MY}. Indeed, for  $\lambda=\tilde{\lambda}$ and $p=q\geq 1$ we obtain the Theorem B of \cite{MY}.
\end{rem}

	\bibliographystyle{plain}
	
\end{document}